\documentclass{article}
\title{Coherent forests}
\author{Monroe Eskew}
\date{}
\usepackage{color} 
\usepackage{amssymb} 
\usepackage{amsmath} 
\usepackage[utf8]{inputenc} 
\usepackage{enumerate}

\DeclareMathOperator{\dom}{dom}
\DeclareMathOperator{\ran}{ran}

\DeclareMathOperator{\cf}{cf}

\newtheorem{theorem}{Theorem}[section]
\newtheorem{lemma}[theorem]{Lemma}
\newtheorem{proposition}[theorem]{Proposition}
\newtheorem{corollary}[theorem]{Corollary}

\newenvironment{proof}[1][Proof]{\begin{trivlist}
\item[\hskip \labelsep {\bfseries #1}]}{\end{trivlist}}
\newenvironment{definition}[1][Definition]{\begin{trivlist}
\item[\hskip \labelsep {\bfseries #1}]}{\end{trivlist}}

\newenvironment{remark}[1][Remark]{\begin{trivlist}
\item[\hskip \labelsep {\bfseries #1}]}{\end{trivlist}}

\newenvironment{question}[1][Question]{\begin{trivlist}
\item[\hskip \labelsep {\bfseries #1}]}{\end{trivlist}}

\begin{document}
\maketitle

\begin{abstract}
A forest is a generalization of a tree, and here we consider the Aronszajn and Suslin properties for forests.  We focus on those forests satisfying coherence, a local smallness property.  We show that coherent Aronszajn forests can be constructed within ZFC.  We give several ways of obtaining coherent Suslin forests by forcing, one of which generalizes the well-known argument of Todor\v{c}evi\'{c} that a Cohen real adds a Suslin tree.  Another uses a strong combinatorial principle that plays a similar role to diamond.  We show that, starting from a large cardinal, this principle can be obtained by a forcing that is small relative to the forest it constructs.
\end{abstract}

We consider a type of structure called a forest, a generalization of a tree.  Forests contain many trees, but can be much wider than a single tree.  Thomas Jech had previously studied the same type of object under the name ``mess'' \cite{jech73}.  The nicer choice of terminology is due to Christoph Weiß \cite{weiss}.  In contrast to the work of Weiß, we will focus on forests that do not contain long branches.

The notions of being Aronszajn and Suslin carry over from trees to forests.  In this paper, we explore several ways of obtaining large Aronszajn and Suslin forests that also satisfy a certain local smallness property called coherence.  We show that large coherent Aronszajn forests can be constructed within ZFC and by forcing.  Next, we explore a constraint imposed by the P-ideal dichotomy that shows the optimality of some of these results.  Finally, we give three ways of forcing large coherent Suslin forests.  The first is a modification of Jech's method of forcing by local approximations.  The second generalizes the argument of Todor\v{c}evi\'{c} that a Cohen real adds a Suslin tree.  Here, we compose a Cohen-generic function with certain kind of coherent Aronszajn forest, and show that while the structure remains non-trivial, all large antichains destroyed.  The third method uses a guessing principle that plays a similar role to diamond in the construction of Suslin trees.  We show that this principle can be obtained from a Mahlo cardinal $\kappa$ using a forcing of size $\kappa$, yet results in a coherent Suslin forest of size $2^\kappa$.  In other work \cite{eskew}, this last result is applied to the study of saturated ideals.

\begin{definition}
A $(\kappa,X,\mu)$-\emph{forest} is a collection of functions $F$ satisfying:
\begin{enumerate}[(1)]
\item $\{ \dom(f) : f \in F \} = \mathcal{P}_\kappa(X)$.
\item $(\forall f \in F) \ran(f) \subseteq \mu$.
\item For $z \in \mathcal{P}_\kappa(X)$, let $F_z = \{ f \in F : \dom(f) = z \}$.  A forest must satisfy that for $z_0 \subseteq z_1$ in $\mathcal{P}_\kappa(X)$, $F_{z_0} = \{ f \restriction z_0 : f \in F_{z_1} \}$.
\end{enumerate}
\end{definition}

Forests are full of trees.  If $F$ is a $(\kappa,X,\mu)$-forest, and $S = \{x_\alpha : \alpha < \kappa \}$ is an enumeration of distinct elements of $X$, then $T_S = \{ f \in F : (\exists \beta < \kappa) \dom(f) = \{ x_\alpha : \alpha < \beta \} \}$ forms a tree of height $\kappa$ under the subset ordering.

A $(\kappa,X,\mu)$-forest $F$ is called \emph{thin} if for all $z \in \mathcal{P}_\kappa(X)$, $|F_z| < \kappa$.  A collection of functions $F$ is called $\kappa$-\emph{coherent} if for all $f,g \in F$, $|\{ x \in \dom(f) \cap \dom(g) : f(x) \not= g(x) \}| < \kappa$.  If $F$ is a $(\kappa^+,X,\mu)$-forest we say it is  \emph{coherent} if it is $\kappa$-coherent.  Clearly, if $\mu \leq \kappa = \kappa^{<\kappa}$, then any coherent $(\kappa^+,X,\mu)$-forest is thin.

A \emph{chain} in a forest is a subset which is linearly ordered under $\subseteq$.  Two elements $f,g$ in a forest $F$ are said to be \emph{compatible} when they have a common extension $h \in F$.  An \emph{antichain} in a forest is a subset of pairwise incompatible elements.  We say that a $(\kappa,X,\mu)$-forest $F$ is \emph{Aronszajn} if it contains no well-ordered chain of length $\kappa$.  We say it is \emph{Suslin} if it contains no antichain of cardinality $\kappa$.  If $F$ is a $(\kappa,X,\mu)$-forest with $\mu \geq 2$, closed under finite modifications, then $F$ is Suslin only if it is Aronszajn.  This is because we can ``split off'' from any chain of length $\kappa$ to get an antichain of size $\kappa$.







\begin{proposition}
If $F$ is a $(\kappa,X,\mu)$-forest, then for any $z \in \mathcal{P}_\kappa(X)$, $F_z$ is a maximal antichain.
\end{proposition}

\begin{proof}
Let $f \in F$, $z \in \mathcal{P}_\kappa(X)$.  By clause (3) of the definition of forests, there is $g \in F$ such that $f \subseteq g$ and $\dom(g) = \dom(f) \cup z$.  Then $g \restriction z \in F_z$, so $g$ is a common extension of $f$ and something in $F_z$.  $\square$
\end{proof}

The following lemma will be useful in several constructions:

\begin{lemma}
\label{agree}
Suppose $F$ is a coherent $(\kappa^{+},X,\mu)$-forest, and $F$ is closed under $< \! \kappa$ modifications.  Then two functions in $F$ have a common extension in $F$ if and only if they agree on their common domain.
\end{lemma}

\begin{proof}
Let $f,g \in F$ agree on $\dom(f) \cap \dom(g)$.  Let $h \in F$ be such that $\dom(h) = \dom(f) \cup \dom(g)$.  By coherence, we can change the values of $h$ on a set of size $< \! \kappa$ to get $h' : \dom(h) \to \mu$ with $h' \restriction \dom(f) = f$, and $h' \restriction \dom(g) = g$.  By the closure of $F$, $h' \in F$.  $\square$
\end{proof}

\section{Aronszajn forests}

The first theorem of this section generalizes of an argument of Koszmider \cite{kosz}.

\begin{lemma}
\label{ext}
Let $\kappa$ be a regular cardinal, and suppose $F = \{ f_\alpha : \alpha < \kappa \}$ is a $\kappa$-coherent set of partial functions from $\kappa$ to $\mu$.
\begin{enumerate}[(a)]
\item There is a function $f: \kappa \to \mu$ such that $\{ f \} \cup F$ is $\kappa$-coherent.
\item If $\mu = \kappa$ and each $f_\alpha$ is $< \! \kappa$ to 1, then there is a $< \! \kappa \!$ to 1 function $f : \kappa \to \kappa$ such that $\{ f \} \cup F$ is $\kappa$-coherent.
\end{enumerate}
\end{lemma}

\begin{proof}
For each $\alpha$, let $D_\alpha = \dom(f_\alpha) \setminus \bigcup_{\beta < \alpha} \dom(f_\beta)$.  Let $E = \kappa \setminus \bigcup_\alpha D_\alpha$.  For the first claim, choose any function $g : E \to \mu$, and let
\[ f(\beta) =
\begin{cases}
f_\alpha(\beta)	&\mbox{if } \beta \in D_\alpha \\
g(\beta)			&\mbox{if } \beta \in E
\end{cases}
\]
For any $\alpha$, $\{ \beta : f(\beta) \not= f_\alpha(\beta) \} = \bigcup_{\gamma < \alpha} \{ \beta \in D_\gamma \cap \dom(f_\alpha) : f_\gamma(\beta) \not = f_\alpha(\beta) \}$.  This is a union of $< \! \kappa$ sets of size $< \! \kappa$, so has size $< \! \kappa$.

For the second claim, choose any $< \! \kappa$ to 1 function $g : E \to \kappa$, and let
\[ f(\beta) =
\begin{cases}
\max(\alpha,f_\alpha(\beta))	&\mbox{if } \beta \in D_\alpha \\
g(\beta)						&\mbox{if } \beta \in E
\end{cases}
\]
For any $\alpha$, $\{ \beta : f(\beta) \not= f_\alpha(\beta) \} \subseteq \bigcup_{\gamma \leq \alpha} \{ \beta \in D_\gamma : f_\gamma(\beta) < \gamma$ or $f_\gamma(\beta) \not = f_\alpha(\beta) \}$.  By the hypotheses, this set has size $< \! \kappa$.  For each $\alpha$, $f^{-1}(\alpha) \subseteq g^{-1}(\alpha) \cup \bigcup \{ f^{-1}_\gamma(\beta) : \gamma,\beta \leq \alpha \}$, so $f$ is $< \! \kappa$ to 1. $\square$
\end{proof}

\begin{theorem}
\label{kos}
Let $\kappa$ be a regular cardinal.  For every $\zeta<\kappa$, there is a coherent $(\kappa^+,\kappa^{+\zeta},\kappa)$-forest consisting of $< \! \kappa$ to 1 functions.
\end{theorem}

\begin{proof}
We will prove by induction the following stronger statement: For every $\zeta < \kappa$ and every sequence $\langle (X_\alpha,F_\alpha) : \alpha < \kappa \rangle$ such that:
\begin{enumerate}[(1)]
\item each $X_\alpha \subseteq \kappa^{+\zeta}$,
\item each $F_\alpha$ is a $(\kappa^+,X_\alpha,\kappa)$-forest of $< \! \kappa$ to 1 functions,
\item $\bigcup_\alpha F_\alpha$ is $\kappa$-coherent,
\end{enumerate}
there is a coherent $(\kappa^+,\kappa^{+\zeta},\kappa)$-forest $F \supseteq \bigcup_\alpha F_\alpha$ consisting of $< \! \kappa$ to 1 functions.

For $\zeta = 0$, pick a collection $\{ f_\alpha : \alpha < \kappa \}$ such that for each $\alpha$, $f_\alpha \in F_\alpha$, and $\dom(f_\alpha) = X_\alpha$.  By Lemma~\ref{ext}(b), there is a $<\kappa$ to 1 function $f : \kappa \to \kappa$ that coheres with each $f_\alpha$, and we can take $F = \{ g : \dom(g) \subseteq \kappa$ and $| \{ x : f(x) \not= g(x) \} | < \kappa \}$.

Assume $\zeta = \eta +1$ and the statement holds for $\eta$.  For each $\beta<\kappa^{+\zeta}$, let $F_\alpha^\beta = \bigcup_\alpha \{ f \restriction \beta : f \in F_\alpha \}$.  We will construct $F \supseteq \bigcup F_\alpha$ as the union of a $\subseteq$-increasing sequence $\langle G_\beta : \beta < \kappa^{+\zeta} \rangle$ such that for each $\beta$, $G_\beta$ is a coherent $(\kappa^+,\beta,\kappa)$-forest of $< \! \kappa$ to 1 functions containing $\bigcup_\alpha F^\beta_\alpha$. Let $G_0 = \{ \emptyset \}$.  Given $G_\beta$, let $G_{\beta+1} = \{ f : \dom(f) \subseteq (\beta+1)$, $\ran(f) \subseteq \kappa$, and $f \restriction \beta \in G_\beta \}$.

Suppose $\beta$ is a limit ordinal of cofinality $\leq \kappa$, and let $\langle \gamma_i : i < \delta \leq \kappa \rangle$ be cofinal in $\beta$.  The collection $\bigcup_{i<\delta} G_{\gamma_i} \cup \bigcup_{\alpha<\kappa} F^\beta_\alpha$ is $\kappa$-coherent, because $(\forall \alpha < \kappa) (\forall f \in F^\beta_\alpha)(\forall i < \delta)(f \restriction \gamma_i \in F^{\gamma_i}_\alpha \subseteq G_{\gamma_i})$.  Since $\beta$ has cardinality $\leq \kappa^{+\eta}$, the inductive assumption implies that we can extend to a forest $G_\beta$ with the desired properties.

Suppose $\beta$ is a limit ordinal of cofinality $> \kappa$.  Let $G_\beta = \bigcup_{\gamma<\beta} G_\gamma$.  Then $G_\beta$ is a forest with the desired properties because $\bigcup_{\alpha<\kappa} F^\beta_\alpha = \bigcup_{\gamma<\beta} (\bigcup_{\alpha<\kappa} F^\gamma_\alpha)$.  Finally, we let $F = \bigcup_{\beta < \kappa^{+\zeta}} G_\beta$.

Now assume $\zeta$ is a limit ordinal of cofinality $<\kappa$, and the statement holds for all $\eta < \zeta$.  Let $\langle \gamma_i : i < \delta = \cf(\zeta) \rangle$ be an increasing cofinal sequence in $\zeta$.  Like above, recursively build an increasing sequence $\langle G_i : i < \delta \rangle$ such that each $G_i$ is a $(\kappa^+,\kappa^{+\gamma_i},\kappa)$-forest of $<\kappa$ to 1 functions extending $\bigcup_\alpha F^{\gamma_i}_\alpha$.  This is done by applying the inductive hypothesis for $\kappa^{+\gamma_i}$ to $\bigcup_\alpha F^{\gamma_i}_\alpha \cup \bigcup_{j<i} G_j$.  We may also assume each $G_i$ is closed under $<\kappa$ modifications.  Simply let $F$ be the collection of functions $f$ such that $\dom(f) \subseteq \kappa^{+\zeta}$, and $(\forall i < \delta) f \restriction \gamma_i \in G_{\gamma_i}$.  Clearly $F \supseteq \bigcup_\alpha F_\alpha$.  

First note that if $f \in F$ were not $< \kappa$ to 1, then there would be some $i < \delta$ such that $f \restriction \kappa^{+\gamma_i}$ is not $< \kappa$ to 1, which is false.  If $f,g \in F$ were to disagree at $\kappa$ many points, then there would be some $i < \delta$ such that $f \restriction \kappa^{+\gamma_i}$ and $g \restriction \kappa^{+\gamma_i}$ disagree at $\kappa$ many points, which is false.  Second, we check that for any $z \in \mathcal{P}_{\kappa^+}(\kappa^{+\zeta})$, there is an $f \in F$ such that $\dom(f) = z$.  We can recursively build a sequence $\langle g_i : i < \delta \rangle$ such that for all $i<j<\delta$, $g_i \in G_i$, $\dom(g_i) = z \cap \kappa^{+\gamma_i}$, and $g_i \subseteq g_j$.  If we have built such a sequence up to $j < \delta$, then $\bigcup_{i<j} g_i \in G_j$, because for any $h \in G_j$ with domain $z \cap \kappa^{+\gamma_j}$, the set of disagreements with $\bigcup_{i<j} g_i$ has size $< \kappa$.  Let $f = \bigcup_{i<\delta} g_i$.  $\square$
\end{proof}

\begin{remark}
Koszmider showed that in the case $\kappa = \omega$, if $\lambda$ is a singular cardinal of cofinality $\omega$, and $\square_\lambda$ and $\lambda^\omega = \lambda^+$ hold, then the induction can push through $\lambda$ as well.  The argument generalizes almost verbatim to show for any regular $\kappa$, the induction can go forward at $\lambda$ of cofinality $\kappa$, under the assumptions $\square_\lambda$ and $\lambda^\kappa = \lambda^+$. (The reader may want to verify this.)  As a consequence, we get that in $L$, for every regular $\kappa$ and every $\lambda \geq \kappa$, there is a coherent, $(\kappa^+,\lambda,\kappa)$-forest of $<\kappa$ to 1 functions.
\end{remark}

Recall that a partial order $\mathbb{P}$ is called $\kappa$-Knaster if for any $A \subseteq \mathbb{P}$ of size $\kappa$, there is $B \subseteq A$ of size $\kappa$ that consists of pairwise compatible elements.

\begin{corollary}For every regular cardinal $\kappa$ and every $\zeta < \kappa$, there is a coherent $(\kappa^+,\kappa^{+\zeta},\kappa)$-forest, which is Aronszajn, does not have the $2^{<\kappa}$ or the $\kappa^+$ chain condition, but is $(2^\kappa)^+$-Knaster.  If $\zeta$ is finite or $2^{<\kappa} < \kappa^{+\omega}$, then the forest is $(2^{<\kappa} \cdot \kappa^{+})^+$-Knaster.
\end{corollary}

\begin{proof}
Let $F$ be given by Theorem~\ref{kos}.  We may assume $F$ is closed under $< \! \kappa$ modifications. To see the failure of the $2^{<\kappa}$ chain condition, note that for any $z \subseteq \kappa^{+\zeta}$ of size $\kappa$, $F_z$ is an antichain of size $2^{<\kappa}$.

Let $\{\alpha_\beta : \beta < \kappa^+ \}$ be any enumeration of distinct ordinals in $\kappa^{+\zeta}$, and for each $\gamma < \kappa^+$, let $f_\gamma \in F$ have domain $\{ \alpha_\beta : \beta < \gamma \}$.  Since each $f \in F$ maps into $\kappa$, there is a $\xi<\kappa$ and a stationary subset $S_0 \subseteq \{ \gamma < \kappa^+ : \cf(\gamma) = \kappa \}$ such that for all $\gamma \in S_0$, $f_{\gamma+1}(\alpha_\gamma) = \xi$.  Since each $f \in F$ is $< \! \kappa$ to 1, each set $\{ \beta < \gamma : f_{\gamma+1}(\alpha_\beta) = \xi \}$ is bounded below $\gamma$ when $\cf(\gamma) = \kappa$.  Thus there is an $\eta < \kappa^+$ and a stationary $S_1 \subseteq S_0$ such that for all $\gamma \in S_1$, $\{ \beta < \gamma: f_{\gamma+1}(\alpha_\beta) = \xi \} \subseteq \eta$.  Therefore, for any $\gamma_0 < \gamma_1$ in $S_1 \setminus \eta$, $f_{\gamma_0+1}(\alpha_{\gamma_0}) \not= f_{\gamma_1+1}(\alpha_{\gamma_0})$.  This shows that $F$ does not have the $\kappa^+$ chain condition.

It also shows that $F$ is Aronszajn.  For otherwise, let $\langle f_\alpha : \alpha < \kappa^+ \rangle$ be a strictly increasing $\subseteq$-chain in $F$.  Let $\{\xi_\beta : \beta < \kappa^+ \} = \bigcup_\alpha \dom(f_\alpha)$, and for each $\gamma$ let $g_\gamma = (\bigcup_\alpha f_\alpha) \restriction \{ \xi_\beta : \beta < \gamma \}$.  Then $\langle g_\gamma : \gamma < \kappa^+ \rangle$ is a strictly increasing chain, but by the above paragraph, it contains an antichain of size $\kappa^+$, contradiction.

To show the $(2^\kappa)^+$-Knaster property, let $\{f_\alpha : \alpha < (2^\kappa)^+ \} \subseteq F$.  Let $T_0 \subseteq (2^\kappa)^+$ have size $(2^\kappa)^+$ and be such that $\{ \dom(f_\alpha) : \alpha \in T_0 \}$ forms a delta-system with root $r$.  Let $T_1 \subseteq T_0$ have size $(2^\kappa)^+$ and be such that for a fixed $g$, $f_\alpha \restriction r = g$ for all $\alpha \in T_1$.  The union of any two of these is in $F$. 

For the case where $\zeta < \omega$ or $2^{<\kappa} < \kappa^{+\omega}$, let $\theta =(2^{<\kappa} \cdot \kappa^{+})^+$.  First note that it is easy to see by induction that for every $n < \omega$, $\mathcal{P}_{\kappa^+}(\kappa^{+n})$ has a cofinal subset of size $\kappa^{+n}$.  Let $A = \{ f_\alpha : \alpha < \theta \} \subseteq F$, and let $S = \bigcup_{\alpha} \dom (f_{\alpha})$.  


Suppose first that $|S| < \theta$.  There is an $R \subseteq \mathcal{P}_{\kappa^+}(S)$ that covers $\{ \dom (f_{\alpha} ): \alpha <\theta \}$ and has cardinality $|S|$.  Therefore, by the coherence of $F$, there is a $G \subseteq F$ of cardinality $\leq |S| \cdot 2^{<\kappa} < \theta$ such that for all $\alpha<\theta$, there is $g \in G$ with $f_{\alpha} \subseteq g$.  Therefore there is a $g_{0} \in G$ which is a common lower bound to $\theta$ many $f_{\alpha}$.

Now suppose that $|S|=\theta$.  Since $\theta$ is regular and $\theta > \kappa^+$, we can use the delta-system argument to get an $S_{0} \subseteq S$ of cardinality less than $\theta$ and a $T_{0} \subseteq \theta$ of cardinality $\theta$ such that for all $\alpha_{0},\alpha_{1} \in T_{0}$, $\dom(f_{\alpha_{0}}) \cap \dom(f_{\alpha_{1}}) \subseteq S_{0}$.  By the above paragraph, there is a $T_{1} \subseteq T_{0}$ of cardinality $\theta$ such that for any $\alpha_{0}, \alpha_{1} \in T_{1}$, $f_{\alpha_{0}}$ and $f_{\alpha_{1}}$ agree on their common domain contained in $S_{0}$. $\square$




\end{proof}

One may ask whether the condition ``$< \! \kappa$ to 1'' can be strengthened to ``1 to 1'' in Theorem~\ref{kos}.  But this cannot always be achieved:

\begin{proposition}
\label{ad}
If there is a coherent $(\kappa^+,\lambda,\kappa)$-forest consisting of injective functions, then there are $\lambda$ many almost disjoint subsets of $\kappa$.
\end{proposition}

\begin{proof}
Let $F$ be such a forest, and for each $z \in \mathcal{P}_{\kappa^+}(\lambda)$, choose $f_z \in F$ with domain $z$.  Let $S$ be a collection of $\lambda$ many pairwise disjoint subsets of $\lambda$, each of cardinality $\kappa$.  For $x \not= y$ in $S$, $\ran(f_x)$ is almost disjoint from $\ran(f_y)$.  This is because the sets $A = \ran(f_{x \cup y} \restriction x)$ and $B = \ran(f_{x \cup y} \restriction y)$ are disjoint, and $|A \triangle \ran(f_x) | < \kappa$, and $|B \triangle \ran(f_y) | < \kappa$.  $\square$
\end{proof}

A positive answer in the following special case is well-known (see \cite{kunen}, Chapter II, Theorem 5.9 and exercise 37):

\begin{theorem}Let $\kappa$ be a regular cardinal.  There is a $\kappa$-coherent collection of functions $\{ f_\alpha : \alpha < \kappa^+ \}$, such that each $f_\alpha$ is an injection from $\alpha$ to $\kappa$.
\end{theorem}

A more general positive answer can be forced:

\begin{theorem}
\label{injections}
Assume $\kappa$ is a regular cardinal with $2^{<\kappa}=\kappa$, and $\lambda \geq \kappa$.  There is a $\kappa$-closed, $\kappa^+$-c.c. partial order that adds a coherent $(\kappa^+,\lambda,\kappa)$-forest of injective functions.
\end{theorem}

\begin{remark}Such a forest will be Aronszajn because a chain of length $\kappa^+$ would give an injection from $\kappa^+$ to $\kappa$.  Unlike the forests of Theorem~\ref{kos}, it will never have the $\lambda$ chain condition.
\end{remark}

\begin{proof}
Let $\mathbb{P}$ be the collection partial functions $p$ that assign to $<\kappa$ many $z \subseteq \lambda$ of size $\leq \kappa$, a partial injective function from $z$ to $\kappa$ defined at $<\kappa$ many points.  Let $p \leq q$ when:
\begin{enumerate}[(a)]
\item $\dom(p) \supseteq \dom(q)$.
\item For all $z \in \dom(q)$, $p(z) \supseteq q(z)$.
\item If $z_0,z_1 \in \dom(q)$, $\alpha \in z_0 \cap z_1 \setminus (\dom(q(z_0)) \cup \dom(q(z_1))$, and $\alpha \in \dom(p(z_0))$, then $\alpha \in \dom (p(z_1))$ and $p(z_0)(\alpha) = p(z_1)(\alpha)$.
\end{enumerate}

It is easy to check that $\leq$ is transitive and that $\langle \mathbb{P},\leq \rangle$ is $\kappa$-closed.  To check the chain condition, let $A \subseteq \mathbb{P}$ have size $\kappa^+$.  Since $\kappa^{<\kappa}=\kappa$, we can find a $B \subseteq A$ of size $\kappa^+$ such that $\{ \dom(p) : p \in B \}$ forms a delta-system with root $R$.  Again since $\kappa^{<\kappa}=\kappa$, there is a $C \subseteq B$ of size $\kappa^+$ and a collection of functions $\{ f_z : z \in R \}$ such that $\forall p \in C$, $\forall z \in R$, $p(z) = f_z$.  If $p,q \in C$, then $p \cup q$ is a common extension.

If $G \subseteq \mathbb{P}$ is generic, then for all $z \in \mathcal{P}_{\kappa^+}(\lambda)^V$, $G$ gives an injective function $f_z: z \to \kappa$ as $\bigcup \{ p(z) : z \in p \in G \}$.  For $z_0, z_1 \in \mathcal{P}_{\kappa^+}(\lambda)^V$, there is some $p \in G$ such that $z_0,z_1 \in \dom(p)$.  $p$ forces that $f_{z_0}$ and $f_{z_1}$ agree outside $\dom(p(z_0)) \cup \dom(p(z_1))$.  Finally, by the $\kappa^+$-c.c., $\mathcal{P}_{\kappa^+}(\lambda)^V$ is cofinal in $\mathcal{P}_{\kappa^+}(\lambda)^{V[G]}$.  So we can define a $(\kappa^+,\lambda,\kappa)$-forest $F$ as $\{ f : f$ is an injection into $\kappa$, $(\exists z) \dom(f) \subseteq z \in \mathcal{P}_{\kappa^+}(\lambda)^V$, and $f$ disagrees with $f_z$ at $< \kappa$ many points$\}$.  $\square$
\end{proof}

\section{Influence of the P-ideal dichotomy}

In the previous section, we saw that coherent, Aronszajn $(\omega_1,\omega_n,\omega)$-forests can be constructed in ZFC for every natural number $n$.  Here we show that the third coordinate is optimal, in the sense that for $n < \omega$ and $\lambda \geq \omega_1$, ZFC cannot prove the existence of a coherent, Aronszajn $(\omega_1,\lambda,n)$-forest.  Let us recall the relevant notions:

\begin{definition}An ideal $I \subseteq \mathcal{P}(X)$ is a \emph{P-ideal} when $\mathcal{P}_\omega(X) \subseteq I \subseteq \mathcal{P}_{\omega_1}(X)$, and for any $\{ z_n : n < \omega \} \subseteq I$, there is $z \in I$ such that $z_n \setminus z$ is finite for all $n$.
\end{definition}

\begin{definition}The \emph{P-ideal dichotomy (PID)} is the statement that for any P-ideal $I$ on a set $X$, either
\begin{enumerate}[(1)]
\item there is an uncountable $Y \subseteq X$ such that $\mathcal{P}_{\omega_1}(Y) \subseteq I$, or
\item there is a partition of $X$ into $\{ X_n : n < \omega \}$ such that for all $n$ and all $z \in I$, $z \cap X_n$ is finite.
\end{enumerate}
\end{definition}

PID is a consequence of the Proper Forcing Axiom, and is also known to be consistent with ZFC+GCH relative to a supercompact cardinal \cite{pideal}. The restriction of PID to ideals on sets of size $\omega_1$ is known to be consistent without the use of large cardinals, both with and without GCH \cite{smallpideal}.

Using a coherent, Aronszajn $(\omega_1,\omega_1,\omega)$-forest $F$, we can obtain a coherent, Aronszajn $\omega_1$-tree $T$ of binary functions by taking the collection of characteristic functions of members of $F$ whose domain is an ordinal, considering the functions as subsets of $\alpha \times \omega$ for $\alpha < \omega_1$.  A cofinal branch would be a function $g : \omega_1 \times \omega \to 2$ with $g \restriction (\alpha \times \omega) \in T$ for all $\alpha < \omega_1$, and this would code an uncountable well-ordered chain in $F$.  Further, using a regressive function argument, we can see that the closure of $T$ under finite modifications remains Aronszajn.  On the other hand, forests are more flexible.  If we take such a tree $T$, close it under subsets to get a forest $F$, then it may be that there is an uncountable well-ordered chain $C \subseteq F$, but with $\dom( \bigcup C)$ a proper subset of $\omega_1 \times \omega$.  This is what happens under PID.

\begin{theorem}Assume PID, and let $F$ be a coherent $(\omega_1,\lambda,n)$-forest closed under finite modifications, for some $\lambda \geq \omega_1$, $n<\omega$.  Then $F$ is not Aronszajn.
\end{theorem}

\begin{proof}
First we prove this for $n = 2$.  Let $F$ be a coherent $(\omega_1,\lambda,2)$-forest closed under finite modifications.  Let $I$ be the collection of $z \subseteq \lambda$ such that for some $f \in F$, $z \subseteq \{ \alpha : f(\alpha) = 1 \}$.

We claim $I$ is a P-ideal.  Let $\{ z_n : n < \omega \} \subseteq I$, and for each $n$, choose $f_n \in F$ witnessing $z_n \in I$.  Let $f \in F$ have domain $\bigcup_n \dom(f_n)$, and let $z = \{ \alpha : f(\alpha) = 1 \}$.  For any $n$, $f$ disagrees with $f_n$ on a finite set, so there can only be finitely many $\alpha \in z_n \setminus z$.

Assume that alternative (1) of PID holds, and let $Y \subseteq \lambda$ be uncountable such that $\mathcal{P}_{\omega_1}(Y) \subseteq I$.  Enumerate $Y$ as $\langle y_\alpha : \alpha < \omega_1 \rangle$.  For each $\alpha < \omega_1$, let $f_\alpha$ be the function that has $f_\alpha(y_\beta) = 1$ for $\beta <\alpha$, and is undefined elsewhere.  Since $F$ is closed under subsets, each $f_\alpha \in F$, and these form an uncountable well-ordered chain.

Assume alternative (2) of PID holds.  Let $X_n \subseteq \lambda$ be uncountable such that for all $z \in I$, $X_n \cap z$ is finite.  Let $g$ have constant value 0 on $X_n$.  If $f \in F$ and $\dom(f) \subseteq X_n$, then $\{ \alpha : f(\alpha) = 1 \}$ is finite.  Thus for any countable $z \subseteq X_n$, $g \restriction z \in F$, so again we have an uncountable well-ordered chain.

Now assume the result holds for $n$, and let $F$ be a coherent $(\omega_1,\lambda,n+1)$-forest.  Let $r(k) = 0$ for $k < n$, and $r(n) = 1$.  Consider the forest $G = \{ r \circ f : f \in F \}$, and let $g_0, g_1$ be the functions on $\lambda$ with constant value 0 and 1 respectively.  By the above argument, there is some uncountable $Y \subseteq \lambda$ such that either $g_0 \restriction z \in G$ for all countable $z \subseteq Y$, or likewise for $g_1$.  The latter case shows that $F$ is not Aronszajn.  In the former case, we have that for all countable $z \subseteq Y$, there is a function $f_z \in F$ with domain $z$ that only takes values below $n$.  If $H = \{ g : (\exists z \in \mathcal{P}_{\omega_1}(Y)) g : z \to n$ and $\{ \alpha : g(\alpha) \not= f_z(\alpha) \}$ is finite$\}$, then $H$ is a coherent $(\omega_1,Y,n)$-forest contained in $F$.  By induction, $H$ contains an uncountable well-ordered chain.  $\square$
\end{proof}

\section{Suslin forests}

\begin{lemma}
Let $\kappa$ be a regular cardinal.  All Suslin $(\kappa,\lambda,\mu)$-forests are $\kappa$-distributive.
\end{lemma}

\begin{proof}
Let $F$ be a Suslin $(\kappa,\lambda,\mu)$-forest, and let $\langle A_{\alpha} : \alpha < \delta < \kappa \rangle$ be a sequence of maximal antichains contained in $F$.  By the Suslin property, each $A_{\alpha}$ has size $< \! \kappa$, so if $z = \bigcup_{\alpha} \{ \dom(f) : f \in A_{\alpha} \}$, $|z| < \kappa$.  By maximality, for every $\alpha< \delta$ and every $g \in F_{z}$, there is an $f \in A_{\alpha}$ such that $g$ is compatible with $f$.  But since $\dom(f) \subseteq \dom(g)$, this means $f \subseteq g$.  Thus $F_{z}$ refines each $A_{\alpha}$.  $\square$
\end{proof}

The boolean completion of a Suslin $(\kappa,\lambda,\mu)$-forest is a $\kappa$-Suslin algebra, which is a complete boolean algebra with that is both $\kappa$-c.c. and $\kappa$-distributive.  The cardinality of this algebra is at least $\lambda$.  Therefore the existence of varieties Suslin forests is constrained by the following (see \cite{jech3}, Theorem 30.20):

\begin{theorem}[Solovay]
If $\mathbb{B}$ is a $\kappa$-Suslin algebra, then $|\mathbb{B} | \leq 2^\kappa$.
\end{theorem}

Large Suslin forests can be obtained by forcing.  In \cite{jech73}, Jech defined a class of partial orders $\mathbb{P}_\lambda$ such that under CH, $\mathbb{P}_\lambda$ is countably closed, $\omega_2$-c.c., and adds a Suslin $(\omega_1,\lambda,2)$-forest.  However, this forest fails to be coherent.  Modifying his forcing slightly, we obtain:

\begin{theorem}Assume $\kappa$ is a regular cardinal, $2^{<  \kappa} = \kappa$, and $2^\kappa = \kappa^+$.  Then for all $\lambda > \kappa$, there is a $\kappa^+$-closed, $\kappa^{++}$-c.c. forcing of size $\lambda^{<\kappa}$ that adds a coherent, Suslin $(\kappa^+,\lambda,2)$-forest.
\end{theorem}

\begin{proof}[Proof (sketch)]
Let $\mathbb{P}$ be the set of all partial functions $f$ from $\lambda$ to 2 of size $\leq \kappa$, and say $f \leq g$ when $\dom(f) \supseteq \dom(g)$ and $| \{ \alpha : f(\alpha) \not= g(\alpha) \} | < \kappa$.  $\kappa^+$-closure follows from Lemma~\ref{ext}(a), and the $\kappa^{++}$-c.c. follows from a delta-system argument.  If $G$ is $\mathbb{P}$-generic over $V$, in $V[G]$ let $F = \{ f : (\exists g \in G) \dom(g) = \dom(f)$ and $| \{ \alpha : f(\alpha) \not= g(\alpha) \} | < \kappa \}$.  Clearly $F$ is coherent.  The argument that $F$ is Suslin in $V[G]$ proceeds as in \cite{jech73}. $\square$
\end{proof}

By adapting an argument of Todor\v{c}evi\'{c} that appears in \cite{todor}, we can obtain large Suslin forests in a different way:

\begin{theorem}
\label{tod}
Assume $\kappa$ is a regular cardinal, $2^{< \kappa} = \kappa$, and there is a coherent $(\kappa^+,\lambda,\kappa)$-forest of injective functions.  Then adding a Cohen subset of $\kappa$ adds a coherent, Suslin $(\kappa^+,\lambda,2)$-forest.
\end{theorem}

\begin{proof}
Let $F$ be a coherent $(\kappa^+,\lambda,\kappa)$-forest of injections closed under $<\kappa$ modifications to other injections.  Let $g : \kappa \to 2$ be an $Add(\kappa)$ generic function over $V$.  Consider the family $G_0 = \{ g \circ f : f \in F \}$.  Since $Add(\kappa)$ is $\kappa^+$-c.c., $\mathcal{P}_{\kappa^+}(\lambda)^V$ is cofinal in $\mathcal{P}_{\kappa^+}(\lambda)^{V[g]}$, so $G_0$ generates a forest $G$ when we close under subsets.  $G$ inherits coherence from $F$.  We claim $G$ is Suslin.

First we note that $G$ is closed under $<\kappa$ modifications.  If $f \in F$, then by the argument for Proposition~\ref{ad}, $\kappa \setminus \ran(f)$ has size $\kappa$.  By a density argument, $\{ \alpha \in \kappa \setminus \ran(f) : g(\alpha) = i \}$ has size $\kappa$ for both $i = 0,1$.  So if $g \circ f \in G$, and $x \subseteq \dom{f}$ has size $<\kappa$, we can switch values of $g \circ f$ on $x$ by choosing distinct ordinals $\{ \alpha_i : i \in x \} \subseteq \kappa \setminus \ran(f)$ such that $g(\alpha_i) = g(f(i)) + 1 \mod 2$.  If $f^\prime = f$ except that $f^\prime(i) = \alpha_i$ for $i \in x$, then $f^\prime \in V$ by $\kappa$-closure, so $g \circ f^\prime \in G$.  So by Lemma~\ref{agree}, members of $G$ have a common extension when they agree on their common domain.

Towards a contradiction, suppose $A = \{ g \circ f_\alpha : \alpha < \kappa^+ \}$ is an antichain in $G_0$, and let $p_0 \in Add(\kappa)$ force this.  Since $|Add(\kappa)| = \kappa$, there is some $p_1 \leq p_0$ such that $p_1 \Vdash \dot{g} \circ \check{f} \in \dot{A}$ for $\kappa^+$ many $f \in F$.  Let $A_0 = \{ f : p_1 \Vdash \dot{g} \circ \check{f} \in \dot{A} \}$, and let $Z = \bigcup \{ \dom(f) : f \in A_0 \}$.

Case 1: $|Z| \leq \kappa$.  Let $h \in F$ be such that $\dom(h) = Z$.  There are at most $\kappa$ many $<\kappa$ modifications of $h$, so there are $f_0, f_1 \in A_0$ such that both agree with the same modification of $h$.  But $p_1$ forces that $g \circ f_0$ and $g \circ f_1$ are compatible, contradiction.

Case 2: $|Z| = \kappa^+$.  Let $\langle \alpha_i : i < \kappa^+ \rangle$ be an enumeration of $Z$.  Let $\beta_0 = \sup(\dom(p_1)) + 1$, and for each $f \in A_0$, let $X_f = \{ \alpha : f(\alpha) < \beta_0 \}$.  Since each $f$ is injective, each $|X_f| <\kappa$.  For each $X_f$, let $\langle X_f(i) : i < \beta_f \rangle$ be an enumeration of $X_f$ that agrees in order with the above enumeration of $Z$.

Case 2a: There is no $i < \kappa$ such that $| \{ X_f(i) : f \in A_0 \} | = \kappa^+$.  Then there is a $\gamma < \kappa^+$ such that for all $f \in A_0$, $\{ i : \alpha_i \in X_f \} \subseteq \gamma$.  Since $\kappa^{<\kappa} = \kappa$, we may choose some $A_1 \subseteq A_0$ such that for all $f \in A_1$, $X_{f}$ is the same set $S$, and further that $f \restriction S$ is the same for all $f \in A_1$.

Let $f_0,f_1 \in A_1$, and let $D = \{ \alpha \in \dom(f_0) \cap \dom(f_1) : f_0(\alpha) \not= f_1(\alpha) \}$.  $|D|<\kappa$, $D \cap S = \emptyset$, and if $\alpha \in D$, then $f_0(\alpha),f_1(\alpha) \geq \beta_0$.  Thus we can define a $q \leq p_1$ such that for all $\alpha \in D$, $q \circ f_0(\alpha) = q \circ f_1(\alpha) = 0$.  $q$ forces that $g \circ f_0$ and $g \circ f_1$ are compatible, contradiction.

Case 2b: There is some $i < \kappa$ such that $| \{ X_f(i) : f \in A_0 \} | = \kappa^+$.  Let $i_0$ be the least such ordinal.  We choose a sequence $\langle f_\alpha : \alpha < \kappa^+ \rangle$.  Let $f_0 \in A_0$ be arbitrary.  Let $f_1$ be such that $X_{f_1}(i_0)$ has index in the enumeration of $Z$ above $\{ i : \alpha_i \in \dom(f_0) \}$.  Keep going in this fashion such that for $\beta < \gamma < \kappa^+$, $X_{f_\gamma}(i_0)$ has index greater than $\sup \{ i : \alpha_i \in \dom(f_\beta) \}$.  By the minimality of $i_0$, there is $C \subseteq \kappa^+$ of size $\kappa^+$ and a set $S \subseteq Z$ such that for all $\alpha \in C$, $\{ X_{f_\alpha}(i) : i < i_0 \} = S$, and $f_\alpha \restriction S$ is the same.

Now let $\beta < \gamma$ be in $C$, and let $D = \{ \alpha \in \dom(f_\beta) \cap \dom(f_\gamma) : f_\beta(\alpha) \not= f_\gamma(\alpha) \}$.  As before, $|D|<\kappa$ and $D \cap S = \emptyset$.  If $\alpha \in D$, then $f_\gamma(\alpha) \geq \beta_0$, because $X_{f_\gamma} \cap \dom(f_\beta) = S$.  We construct $q \leq p_1$ such that for all $\alpha \in D$, $q \circ f_\gamma(\alpha) = q \circ f_\beta(\alpha)$.  Let $D_{0} = \{ \alpha \in D :  f_{\beta}(\alpha) \in \dom(p_{1}) \}$, and let $q_{0} = p_{1} \cup \{ \langle f_{\gamma}(\alpha), p_{1} \circ f_{\beta}(\alpha) \rangle : \alpha \in D_{0} \}$.  We are free to do this because $f_{\gamma}$ is injective and $f_{\gamma}(\alpha) \notin \dom(p_1)$ for $\alpha \in D$.

Note that for all $\alpha \in D$, $q_0$ is defined at $f_\gamma(\alpha)$, only if it is defined at $f_\beta(\alpha)$.  But it may be that for some $\alpha \in D_0$ and some $\alpha^\prime \in D \setminus D_0$, $f_\gamma(\alpha) = f_\beta(\alpha^\prime)$.  Assume we have a sequence $q_0 \geq  ... \geq q_n$ such that:

\begin{enumerate}[(1)]
\item for all $k \leq n$, $D \cap f^{-1}_\gamma[\dom(q_k)] \subseteq D \cap f^{-1}_\beta[\dom(q_k)]$,
\item for all $k \leq n$, $q_k \circ f_\gamma \restriction (D \cap f^{-1}_\gamma[\dom(q_k)]) = q_k \circ f_\beta \restriction (D \cap f^{-1}_\gamma[\dom(q_k)])$,
\item if $k+1 \leq n$, then $D \cap f^{-1}_\gamma[\dom(q_{k+1})] = D \cap f^{-1}_\beta[\dom(q_k)]$.
\end{enumerate}
If $D \cap f^{-1}_\gamma[\dom(q_n)] = D \cap f^{-1}_\beta[\dom(q_n)]$, let $q_{n+1} = q_n$.  Otherwise, let $D_{n+1} = D \cap f^{-1}_\beta[\dom(q_n)]$, and let $q_{n+1} = q_n \cup \{ \langle f_{\gamma}(\alpha), q_n \circ f_{\beta}(\alpha) \rangle : \alpha \in D_{n+1} \}$.  Clearly the induction hypotheses are preserved for $n+1$.


Put $q_\omega = \bigcup q_n$.  (Note in the case $\kappa= \omega$, $D$ is finite, so $q_\omega = q_n$ for some $n$.)  By (1) and (3), $D \cap f^{-1}_\beta[\dom(q_\omega)] = D \cap f^{-1}_\gamma[\dom(q_\omega)]$, so call this set $D_\omega$.  Let  $q = q_\omega \cup \{ \langle f_\beta(\alpha) , 0 \rangle : \alpha \in D \setminus D_\omega \} \cup \{ \langle f_\gamma(\alpha) , 0 \rangle : \alpha \in D \setminus D_\omega \}$.  This $q$ forces that $g \circ f_\beta$ and $g \circ f_\gamma$ are compatible, again in contradiction to the assumption about $p_1$. $\square$

\end{proof}

\begin{corollary}Assume $\kappa$ is a regular cardinal, $2^{< \kappa} = \kappa$, and $\lambda > \kappa$.  Then there is a $\kappa$-closed, $\kappa^+$-c.c. forcing that adds a coherent, Suslin $(\kappa^+,\lambda,2)$-forest.
\end{corollary}

\begin{proof}Apply Theorems~\ref{injections} and~\ref{tod}.  $\square$
\end{proof}

Large Suslin forests can also be obtained from combinatorial principles rather than forcing.  As reported by Jech \cite{jech1} \cite{jech3} \cite{jech73}, Laver proved in unpublished work that the existence of Suslin $(\omega_1,\omega_2,2)$-forests follows from Silver's principle $W$ and $\diamondsuit$, both of which hold in $L$.  Unfortunately, Laver's proof seems to be lost to history.  In trying to reconstruct it, we encountered technical issues that led to the development of a new combinatorial principle, which we prove consistent from a Mahlo cardinal, that can be used to construct large Suslin forests.  The main appeal for us is that, unlike the above forcing constructions, it allows a Suslin $(\kappa,\kappa^+,2)$-forest to be generically added to any model with sufficiently large cardinals using a forcing of size $\kappa$ rather than $\kappa^+$.




Let us establish some notation concerning trees. Suppose $T$ is a $\kappa$-tree and $\alpha<\kappa$. $T_\alpha$ is the set of nodes at level $\alpha$.  If $b$ is a cofinal branch in $T$, $\pi_\alpha(b)$ is the node at level $\alpha$ in $b$.  If $\beta < \alpha$, and $x \in T_\alpha$, $\pi_{\alpha,\beta}(x)$ is the node in $T_\beta$ below $x$.

\begin{definition}
$W_\kappa(\lambda)$ is the statement that there is a $\kappa$-tree $T$, a set of cofinal branches $B$, and a sequence $\langle W_\alpha : \alpha < \kappa \rangle$ with the following properties:
\begin{enumerate}[(1)]
\item $|B| = \lambda$.
\item For each $\alpha$, $|W_\alpha| < \kappa$, and $W_\alpha \subseteq \mathcal{P}(T_\alpha)$.
\item For every $z \in \mathcal{P}_\kappa(B)$, there is an $\alpha < \kappa$ such that for all $\beta \geq \alpha$, $\pi_\beta[z] \in W_\beta$.
\end{enumerate}
\end{definition}

Let $T$, $B$, $\langle W_\alpha : \alpha < \kappa \rangle$ be as above.  If $z \in \mathcal{P}_\kappa(B)$, say ``\emph{$z$ is captured at $\alpha$}'' when for all $\beta \geq \alpha$, $\pi_\beta[z] \in W_\beta$ and $\pi_\beta \restriction z$ is injective.  If $z \in W_\alpha$ and $\gamma < \alpha$, say ``\emph{$z$ is captured at $\gamma$}'' when for all $\beta$ such that $\gamma \leq \beta < \alpha$, $\pi_{\alpha,\beta}[z] \in W_\beta$ and $\pi_{\alpha,\beta} \restriction z$ is injective.

\begin{definition}$W_{\kappa}^*(\lambda)$ asserts $W_{\kappa}(\lambda)$, and that there exists a stationary $S \subseteq \kappa$ and a sequence $\langle A_\alpha : \alpha < \kappa \rangle$ with each $A_\alpha \subseteq W_\alpha^2$, such that the following additional clauses hold:

\begin{enumerate}[(1)]
\setcounter{enumi}{3}
\item $\kappa = \mu^+$ for a regular cardinal $\mu$, and each $W_\alpha$ is a $\mu$-complete subalgebra of $\mathcal{P}(T_\alpha)$ containing all singletons. 

\item For all $\alpha \in S$, $\{ z \in W_\alpha : z$ is captured below $\alpha \}$ is closed under arbitrary $<\mu$ sized unions and taking subsets which are in $W_\alpha$.

\item If $f : \kappa \to \mathcal{P}_\kappa(B)^2$ is such that $| \bigcup_{\alpha < \kappa} f_0(\alpha) \cup f_1(\alpha)| = \kappa$, let $\langle b_\alpha : \alpha < \kappa \rangle$ enumerate the elements of $\bigcup_{\alpha < \kappa} f_0(\alpha) \cup f_1(\alpha)$.  The set of $\alpha \in S$ with the following properties is stationary:
\begin{enumerate}[(a)]
\item $\{ b_\beta : \beta < \alpha \}$ is captured at $\alpha$.
\item If $z \subseteq \{ \pi_\alpha(b_\beta) : \beta < \alpha \}$ is captured below $\alpha$, then $\sup \{ \beta : \pi_\alpha(b_\beta) \in z \} < \alpha$.
\item $\{ \langle \pi_\alpha[ f_0(\beta) ],\pi_\alpha[ f_1(\beta) ] \rangle : \beta < \alpha \} = A_\alpha$.
\end{enumerate}

\end{enumerate}
\end{definition}

\begin{remark}
It is easy to see that $W_\kappa(\lambda)$ implies $2^{<\kappa} = \kappa$, and in fact $W_\kappa(\kappa)$ is equivalent to $2^{<\kappa} = \kappa$. If $\kappa = \mu^+$ and $S$ forms part of the witness to $W^*_\kappa(\lambda)$, then clause (4) implies $\mu^{<\mu} = \mu$, and clause (6) can be used to show $\diamondsuit_\kappa(S)$.  On the other hand, it follows from the next theorem that $W^*_\kappa(\lambda)$ prescribes no value for $2^\kappa$, besides that $\lambda \leq 2^\kappa$.
\end{remark}

\begin{theorem}
Suppose $\kappa$ is a Mahlo cardinal and $\mu < \kappa$ is regular.  If $G * H \subseteq Col(\mu,<\kappa) * Add(\kappa)$ is generic, then $V[G*H]$ satisfies $W^*_\kappa(2^\kappa)$.
\end{theorem}

\begin{proof}
In $V$, let $T$ be the complete binary tree on $\kappa$, and let $B$ be the set of all branches.  For $\alpha < \kappa$, let $G_\alpha = G \cap Col(\mu,<\alpha)$, and let $W_\alpha = \mathcal{P}(T_\alpha)^{V[G_\alpha]}$.  Let $S = \{ \alpha < \kappa : \alpha$ is inaccessible in $V \}$.  In $V[G]$, fix enumerations $\langle s^\alpha_\beta : \beta < \mu \rangle$ of the $W_\alpha^2$, and in $V[G*H]$, let $A_\alpha = \{ s^\alpha_\beta : H(\alpha+\beta) = 1 \}$.  Let us check each condition.  

\begin{enumerate}[(1)]

\item $(2^\kappa)^V = (2^\kappa)^{V[G*H]}$, so $V[G*H] \vDash |B| = 2^\kappa$.

\item Since $\kappa$ is inaccessible, each $W_\alpha$ is collapsed to $\mu$.

\item Suppose $z \in \mathcal{P}_\kappa(B)$.  There is some $\alpha < \kappa$ such that $z \in V[G_\alpha]$.  For $\beta \geq \alpha$, $\pi_\beta[z] \in W_\beta$.

\item The regularity of $\mu$ is preserved, and clearly each $W_\alpha$ contains all singletons.  Let $\langle a_\xi : \xi < \delta \rangle \subseteq W_\alpha$ with $\delta < \mu$. Each $a_\xi \in A$ is $\tau_\xi^{G_\alpha}$ for some $Col(\mu,<\alpha)$-name $\tau_\xi$.  By the $\mu$-closure of $Col(\mu,<\kappa)$, $\langle \tau_\xi : \xi < \delta \rangle \in V$, so $\langle a_\xi : \xi < \delta \rangle \in V[G_\alpha]$.

\item By the Mahlo property, $S$ is stationary, and by the $\kappa$-c.c. of $Col(\mu,<\kappa)$ and $\kappa$-closure of $Add(\kappa)$, it remains stationary in $V[G*H]$.  Suppose $\alpha \in S$.

\begin{enumerate}
\item Unions: Let $A \in \mathcal{P}_\mu(W_\alpha)$ have the property that all $a$ in $A$ are captured below $\alpha$.  As above, $A \in V[G_\alpha]$.  Now in $V[G_\alpha]$, $\alpha = \mu^+$ and $|T_\beta| = \mu$ for $\beta< \alpha$.  So if $\pi_{\alpha,\beta} \restriction a$ is injective, then $V[G_\alpha] \vDash |a| < \alpha$, and thus $V[G_\alpha] \vDash | \bigcup A | < \alpha$.  For distinct $x,y \in \bigcup A$, let $\gamma_{x,y} < \alpha$ be the least $\gamma$ such that $\pi_{\alpha,\gamma}(x) \not= \pi_{\alpha,\gamma}(y)$. We have $\gamma = \sup\{\gamma_{x,y} : x,y \in \bigcup A \} < \alpha$.  Hence if $\gamma \leq \beta < \alpha$ and all $a \in A$ are captured at $\beta$, then $\bigcup A$ is captured at $\beta$.
\item Subsets: Suppose $z_0 \in W_\alpha$ is captured below $\alpha$, and $z_1 \in W_\alpha$ is a subset of $z_0$.  Then $V[G_\alpha] \vDash |z_1| < \alpha$, so by the $\alpha$-c.c. of $Col(\mu,<\alpha)$, there is some $\beta < \alpha$ such that $z_1 \in V[G_\beta]$.  Thus $z_1$ is captured below $\alpha$.
\end{enumerate}

\item First work in $V[G]$.  Let $\dot{f}$ be an $Add(\kappa)$-name for a function from $\kappa$ to $\mathcal{P}_\kappa(B)^2$, and let $\langle \dot{b}_\alpha : \alpha < \kappa \rangle$ be as in clause (6).  Let $\dot{C}$ be a name for a club, and let $p_0 \in Add(\kappa)$ be arbitrary.  Build a continuous decreasing chain of conditions below $p_0$, $\langle p_\alpha : \alpha < \kappa \rangle \subseteq Add(\kappa)$, and a continuous increasing chain of ordinals, $\langle \xi_\alpha : \alpha < \kappa \rangle \subseteq \kappa$, with the following properties:   For all $\alpha$,
\begin{itemize}
\item $p_{\alpha+1} \Vdash \xi_\alpha \in \dot{C}$,
\item $p_{\alpha+1}$ decides $\dot{f} \restriction \dom(p_\alpha)$ and $\{ \dot{b}_\beta : \beta < \alpha \}$,
\item $\dom(p_{\alpha+1})$ is an ordinal $> \max \{ \dom(p_\alpha),\xi_\alpha, \alpha \}$, and
\item $\xi_{\alpha+1} > \dom(p_{\alpha+1})$.
\end{itemize}

Let $g : \kappa \to \mathcal{P}_\kappa(B)^2$ and $\{ b_\alpha : \alpha < \kappa \}$ be the objects defined by what the chain $\langle p_\alpha : \alpha < \kappa \rangle$ decides.  For each $\alpha < \kappa$, there is a predense set $E_\alpha \subseteq Col(\mu,<\kappa)$ of size $<\kappa$ such that $g(\alpha)$ and $b_\alpha$ are decided by elements of $E_\alpha$.  There is a club $D \in V$ such that $\forall \alpha \in D$, $\forall \beta < \alpha$, $E_\beta \subseteq Col(\mu,<\alpha)$.  For $\alpha \in D$, $g \restriction \alpha$ and $\{ b_\beta : \beta < \alpha \}$ are in $V[G_\alpha]$.




Back in $V[G]$, for $\alpha < \kappa$, let $\gamma_\alpha$ be the least $\gamma \geq \alpha$ such that $\pi_{\gamma_\alpha} \restriction \{ b_\beta : \beta < \alpha \}$ is injective.  If $\alpha$ is closed under $\beta \mapsto \gamma_\beta$, then $\gamma_\alpha = \alpha$.  As $S$ is stationary, there is $\alpha \in S \cap D$ such that $\gamma_\alpha = \alpha$, $\xi_\alpha = \alpha$, and $p_\alpha \Vdash \alpha \in \dot{C}$.  We have that $\{ b_\beta : \beta < \alpha \}$  is captured at $\alpha$, and that $\{ \langle \pi_\alpha [g_0(\beta)],\pi_\alpha [g_1(\beta)] \rangle : \beta < \alpha \} \subseteq W_\alpha^2$.  Since $\alpha$ is inaccessible in $V$, if $z \subseteq \{\pi_\alpha(b_\beta) : \beta < \alpha \}$ is captured below $\alpha$, then $V[G_\alpha] \vDash |z| < \alpha$, so $\{ \beta : \pi_\alpha(b_\beta) \in z \}$ is bounded below $\alpha$.


Let $q \leq p_\alpha$ be such that for $\beta <\mu$, $q(\alpha + \beta) = 1$ if $s^\beta_\alpha = \langle \pi_\alpha [g_0(\beta)],\pi_\alpha [g_1(\beta)] \rangle$, and $q(\alpha + \beta) = 0$ otherwise.  Then $q \Vdash \alpha \in \dot{C} \cap S$, and that items (a), (b), and (c) in clause (6) hold at $\alpha$.  As $p_0$ was arbitrary, clause (6) is forced.  $\square$



\end{enumerate}
\end{proof}

\begin{question}
Can $W^*_\kappa(\lambda)$ be forced without the use of large cardinals?  Can it be forced in a cardinal-preserving way?  Does $L \vDash$ ``For all regular $\kappa$, $W^*_{\kappa^+}(\kappa^{++})$''?
\end{question}

\begin{theorem}
$W^*_\kappa(\lambda)$ implies there is a coherent, Suslin $(\kappa,\lambda,2)$-forest.
\end{theorem}

\begin{proof}
Let $\kappa = \mu^+$, and let $T$, $B$, $\langle W_\alpha : \alpha < \kappa \rangle$, $\langle A_\alpha : \alpha < \kappa \rangle$, and $S \subseteq \kappa$ witness $W^*_\kappa(\lambda)$.  We will construct a sequence of functions $\langle f_\alpha : \alpha < \kappa \rangle$  on the nodes of $T$ that will generate a coherent family of functions on $B$ with the desired properties.  Each $f_\alpha$ will have domain $T_\alpha$ and range contained in $\{0,1\}$.  



Let $f_0$ be a function from $T_0$ to 2.  Assume we have have constructed a sequence of functions $\langle f_\beta : \beta < \alpha \rangle$, with each $f_\beta : T_\beta \to 2$, satisfying the following property:

\begin{enumerate}[($*$)]
\item If $r \in W_\beta$ is captured at $\gamma<\beta$, then $f_\beta \restriction r$ disagrees with $f_\gamma  \circ \pi_{\beta,\gamma} \restriction r$ on a set of size $< \mu$.
\end{enumerate}


Let $R_\alpha = \{ r \in W_\alpha : r$ is caputured below $\alpha \}$.  Consider the set $F_\alpha$ of partial functions on $T_\alpha$ of the form $f_\gamma \circ \pi_{\alpha,\gamma} \restriction r$ for $r \in R_\alpha$ and $\gamma$ witnessing its membership in $R_\alpha$.  Assume $\gamma_0 < \gamma_1$ and $f_{\gamma_0} \circ \pi_{\alpha,\gamma_0} \restriction r_0$ and $f_{\gamma_1} \circ \pi_{\alpha,\gamma_1} \restriction r_1$ are in $F_\alpha$. By hypothesis $(*)$, $f_{\gamma_1}$ disagrees with $f_{\gamma_0} \circ \pi_{\gamma_1,\gamma_0}$ at less than $\mu$ many points in $\pi_{\alpha,\gamma_1}[r_0]$.  Therefore, there are less than $\mu$ many points in $r_0 \cap r_1$ at which $f_{\gamma_0} \circ \pi_{\alpha,\gamma_0}$ and $f_{\gamma_1} \circ \pi_{\alpha,\gamma_1}$ disagree.  So $F_\alpha$ is a $\mu$-coherent family.  


Assume first that $\alpha \notin S$.  Using Lemma~\ref{ext}(a), let $f_\alpha : T_\alpha \to 2$ be such that $\{ f_\alpha \} \cup F_\alpha$ is $\mu$-coherent.  Then $(*)$ holds for $\langle f_\beta : \beta \leq \alpha \rangle$.

Now assume $\alpha \in S$.  Let $H_\alpha$ be the closure of $F_\alpha$ under $< \mu$ modifications. Consider $H_\alpha$ as a partial order with $f \leq g$ iff $f \supseteq g$.  The set $A_\alpha \subseteq W_\alpha^2$ codes a set of relations from subsets of $T_\alpha$ to 2.  If $\langle a_0,a_1 \rangle \in A_\alpha$, construct a relation $h$ by putting $\langle x,i \rangle \in h$ iff $x \in a_i$, and call the set of all such things $A^\prime_\alpha$.  It may be the case that every member of $A^\prime_\alpha$ is a function and a member of $H_\alpha$, and that $A^\prime_\alpha$ is a maximal antichain in $H_\alpha$.  If not, ignore all these considerations, and let $f_\alpha$ be as in the case $\alpha \notin S$, so that $(*)$ is preserved.



Suppose $A^\prime_\alpha$ is a maximal antichain in $H_\alpha$. Enumerate $R_\alpha$ as $\langle r_\beta : \beta < \mu \rangle$.  By clauses (4) and (5) of the definition of $W^*$, $R_\alpha$ is closed under unions of size $<\mu$.  $H_\alpha$ is also a $\mu$-closed partial order.  If $\langle h_i : i < \beta < \mu \rangle$ is a decreasing sequence, then $\bigcup_{i<\beta} \dom(h_i) = r \in R_\alpha$, so let $\gamma$ witness this.   By $(*)$, each $h_i$ disagrees with $f_\gamma \circ \pi_{\alpha,\gamma}$ on a set of size $< \mu$, and so $\bigcup_{i<\beta} h_i$ does as well by the regularity of $\mu$.

Setting $s_\beta = \bigcup_{\xi<\beta} r_\xi$, we have $\langle s_\beta : \beta < \mu \rangle$ is an increasing cofinal sequence in $R_\alpha$.  For $\beta < \mu$, let $\gamma_\beta$ be the least $\gamma < \alpha$ that witnesses $s_\beta \in R_\alpha$.  Let $\langle t_\beta : \beta < \mu \rangle$ enumerate all $<\mu$ sized subsets of $T_\alpha$, such that each subset is repeated $\mu$ many times.  For a partial function $f : T_\alpha \to 2$ and $\beta < \mu$, let $f / t_\beta$ be $f$ with its output values switched at the points in $\dom(f) \cap t_\beta$.

We will define $f_\alpha$ inductively as $\bigcup_{\beta<\mu} h_\beta$.  Let $h_0 = \emptyset$. Assume  $\langle h_i : i < \beta \rangle$ has been chosen so that:
\begin{enumerate}[(1)]
\item for $i < j < \beta$, $h_i \subseteq h_j$;
\item for $i < \beta$, $\dom(h_i) = s_{\xi_i}$ where $\xi_i \geq i$, and $\xi_i > \xi_j$ for $j<i$;
\item for $i < \beta$, there is $a \in A^\prime_\alpha$ such that $h_{i+1} / t_i$ is a common extension of $h_{i} / t_i$ and $a$.
\end{enumerate}
Given $h_i$, there is some $a \in A^\prime_\alpha$ that is compatible with $h_i / t_i$.  Let $\xi_{i+1} > \xi_i$ be such that $s_{\xi_{i+1}} \supseteq \dom(a) \cup s_{\xi_i}$, and let $g \in H_\alpha$ be a common extension of $a$ and $h_i / t_i$ with domain $s_{\xi_{i+1}}$.  Let $h_{i+1} = g / t_i$.  Clearly (1)--(3) are preserved at successor steps.  At limit steps $\beta$, we set $h_\beta = \bigcup_{i<\beta} h_i$.  This is in $H_\alpha$ as well by $\mu$-closure, and the preservation of (1)--(3) is trivial.



The point is this: For every $t \in \mathcal{P}_\mu(T_\alpha)$, $f_\alpha / t$ extends some $a \in A^\prime_\alpha$.  For let $i < \mu$ be large enough that $s_{\xi_i} \supseteq t$ and $t_i = t$.  Then by (3), $h_{i+1}/t$ extends some $a \in A^\prime_\alpha$, and $h_{i+1}/t = (f_\alpha / t) \restriction s_{\xi_{i+1}}$.  We also check that $(*)$ is preserved at $\alpha$:  Every $r \in R_\alpha$ is covered by some $s_{\xi_i}$, and $f_\alpha \restriction s_{\xi_i} = h_i$, which coheres with $f_\gamma \circ \pi_{\alpha,\gamma} \restriction s_{\xi_i}$ when $s_{\xi_i}$ is captured at $\gamma$.

Now we define the forest.  For $z \in \mathcal{P}_\kappa(B)$, let $\gamma_z$ be the least $\gamma < \kappa$ such that $z$ is captured at $\gamma$.  Let $f_z : z \to 2$ be $f_{\gamma_z} \circ \pi_{\gamma_z} \restriction z$. Let $F$ be the closure of $\{ f_z : z \in \mathcal{P}_\kappa(B) \}$ under $< \mu$ modifications.  Note that by $(*)$, if $\beta \geq \gamma_z$, then $f_{\beta} \circ \pi_{\beta} \restriction z$ disagrees with $f_z$ at $< \mu$ many points.  Hence $F$ is a coherent $(\kappa,B,2)$-forest.

Finally, we verify the $\kappa$-c.c.  First note that $F$ satisfies the $\kappa^+$-c.c. by a delta-system argument.  So assume towards a contradiction that $A = \{ a_\alpha : \alpha < \kappa \}$ is a maximal antichain.  Let $z_\alpha = \dom(a_\alpha)$, and code each $a_\alpha$ as $\langle z_\alpha^0 , z_\alpha^1 \rangle$, where $z_\alpha^i = \{ b : a_\alpha(b) = i \}$.  Let $\langle b_\alpha : \alpha < \kappa \rangle$ enumerate the elements of $\bigcup_{\alpha < \kappa} z_\alpha$.  Define:

\begin{itemize}


\item $C_0 = \{ \alpha < \kappa : \bigcup_{\beta < \alpha} z_\beta = \{ b_\beta : \beta < \alpha \} \}$.

\item $C_1 = \{ \alpha < \kappa : \{ a_\beta : \beta < \alpha \}$ is a maximal antichain contained in $\{ f \in F : (\exists \eta < \alpha) \dom(f) \subseteq \{ b_\beta : \beta < \eta \} \} \}$.

\item $C_2 = \{ \alpha < \kappa : (\forall \beta < \alpha) \gamma_{z^0_\beta},\gamma_{z^1_\beta},\gamma_{z_\beta} < \alpha \}$.


\end{itemize}




It is easy to see that $C_0$, $C_1$, and $C_2$ are club. By clause (6) of the definition of $W^*$, let $\alpha \in S \cap C_0 \cap C_1 \cap C_2$ be such that $\{ b_\beta : \beta < \alpha \} = \bigcup_{\beta<\alpha} z_\beta$ is captured at $\alpha$, all $z \subseteq \{ \pi_\alpha(b_\beta) : \beta < \alpha \}$ captured below $\alpha$ have $\sup \{ \beta : \pi_\alpha(b_\beta) \in z \} < \alpha$, and  $A_\alpha = \{ \langle \pi_\alpha[z^0_\beta], \pi_\alpha[z^1_\beta] \rangle : \beta < \alpha \}$.

We claim $A^\prime_\alpha$ is a maximal antichain in $H_\alpha$.  For $\beta< \alpha$, $z_\beta$ is captured below $\alpha$ since $\alpha \in C_2$, so the function coded by $\langle \pi_\alpha[z^0_\beta], \pi_\alpha[z^1_\beta] \rangle$ is in $H_\alpha$.  If $h \in H_\alpha$ is incompatible with every member of $A^\prime_\alpha$, then consider $z = \{ b_\beta : \beta<\alpha$ and $\pi_\alpha(b_\beta) \in \dom(h) \}$, and let $f = h \circ \pi_\alpha \restriction z$.  Clauses (4) and (5) imply $\pi_\alpha[z]$ is captured below $\alpha$, so $\sup \{ \beta : b_\beta \in z \} < \alpha$.  Since $\alpha \in C_1$, $f$ is compatible with some $a_\beta$ with $\beta < \alpha$.  But $a_\beta$ is coded and projected down as $\langle \pi_\alpha[z^0_\beta], \pi_\alpha[z^1_\beta] \rangle \in A_\alpha$, so $h$ is compatible with some member of $A^\prime_\alpha$ after all.

Since $\{ b_\beta : \beta < \alpha \}$ is captured at $\alpha$, the construction has sealed this antichain.  Consider any other $f \in F$ such that $\dom(f) \supseteq \{ b_\beta : \beta < \alpha \}$.  Then $f \restriction \{ b_\beta : \beta < \alpha \}$ is a $<\mu$ modification of $f_\alpha \circ \pi_\alpha \restriction \{ b_\beta : \beta < \alpha \}$.  By the above argument, all $<\mu$ modifications of $f_\alpha$ extend a member of $A^\prime_\alpha$, and so $f$ is compatible with some $a_\beta$, $\beta<\alpha$.  This contradicts the assumption that $A = \{ a_\gamma : \gamma < \kappa \}$ is an antichain.  $\square$

\end{proof}

\end{document}